\documentclass{amsart}

\usepackage[T1]{fontenc}
\usepackage{amsfonts, amscd, amsmath, amssymb, tikz-cd, csquotes}
\usepackage[hidelinks]{hyperref}

\newcommand{\C}{\mathbb{C}}
\newcommand{\F}{\mathbb{F}}
\newcommand{\Q}{\mathbb{Q}}
\newcommand{\Z}{\mathbb{Z}}

\newcommand{\pr}{{\mathrm{pr}}}
\newcommand{\id}{\mathrm{id}}

\DeclareMathOperator{\Aut}{Aut}
\DeclareMathOperator{\End}{End}
\DeclareMathOperator{\Gal}{Gal}
\DeclareMathOperator{\GL}{GL}
\DeclareMathOperator{\PGL}{PGL}
\DeclareMathOperator{\PSL}{PSL}
\DeclareMathOperator{\SL}{SL}
\DeclareMathOperator{\Res}{Res}

\newtheorem{thm}{Theorem}
\newtheorem{prop}[thm]{Proposition}
\newtheorem{lmm}[thm]{Lemma}

\theoremstyle{remark}
\newtheorem*{ex}{Example}

\title{Descent for projective twists of modular curves}
\author{Franciszek Knyszewski}
\email{franciszek.knyszewski@stcatz.ox.ac.uk}
\address{St Catherine's College, University of Oxford, Oxford OX1 3UJ, United Kingdom}
\date{\today}
\begin{document}
    \begin{abstract} Let $F$ be a number field and $p\geq 7$ a rational prime. We obtain a simple descent criterion characterising those projective Galois representations $\overline\rho:G_F\to \PGL_2(\F_p)$ for which the corresponding twist $X_{\overline\rho}(p)$ of the principal modular curve of level $p$ is defined over $\Q$. We also give a more concrete version of this result for representations which arise from  elliptic curves over cyclic number fields.
\end{abstract}
\maketitle
\tableofcontents
\section{Introduction}
\subsection{Background}\label{1.1} Many successful approaches to the Modularity Conjecture rely at least in part on having good control over the set of rational points on certain modular curves. One has to look no further than Wiles' magnum opus \cite{wiles1995modular}, where the \enquote*{3-5 trick} (used in tandem with the Langlands-Tunnell theorem) allows for a particular instance of Modularity Lifting to establish Fermat's Last Theorem. Generalisations of this technique have since found many more striking applications (see \cite{taylor2002remarks} for instance), and we refer the reader to \cite{buzzard2012potential} for a remarkably lucid survey on the subject.

More recently, Freitas, Le Hung and Siksek \cite{freitas2015elliptic} were able to use an effective version of the aforementioned trick in their work on the modularity of elliptic curves over real quadratic fields. This proceeds by fixing an elliptic curve $E$ over a given real quadratic field $F$, and looking at the points of $X_E(7)$ defined over certain totally real quartic extensions of $F$ (see §7 in loc. \!cit. \!for details). One can imagine trying to adapt this reasoning to tackle modularity over \text{imaginary} quadratic fields, and indeed Caraiani and Newton explore this possibility in \cite{caraiani2023modularity}, to great success. However, there is a certain technical obstruction which arises in this case, one rooted in the fact that CM-fields, as opposed to totally real fields, are rather unwieldy. 

A natural question presents itself: \textit{for which $E$ is the corresponding twisted curve $X_E(7)$ defined over $\Q$?} For if $E$ was such an elliptic curve then $X_E(7)$ would have enough points over certain degree eight CM-fields so as to guarantee that the modularity switching machinery works. In this paper, we will more generally be concerned with twists of modular curves by arbitrary Galois representations, albeit in the projective setting.

\subsection{Statement of the result}\label{1.2}
Fix a rational prime $p\geq 7$ and denote by $\Q(p)$ the unique quadratic extension of $\Q$ inside the $p$-th cyclotomic field $\Q(\zeta_p)$. Explicitly, this is $\Q(p)=\Q(\sqrt{\pm p})$, with the sign determined by $p\equiv\pm1\pmod{4}$. Note that if
$$\overline\varepsilon_p:G_\Q\to\F_p^{\times}$$ is the mod-$p$ cyclotomic character and $$\overline\varepsilon_p^\pr:G_\Q\to\F_p^\times\to \F_p^\times/\F_p^{\times 2}$$ the corresponding projective character then $\Q(p)$ is the number field cut out by $\overline\varepsilon_p^\pr$.

For the convenience of the reader, we briefly recall the relevant aspects of the theory of modular curves. Let us start with the complex picture. The group
\[\varGamma(p)=\Big\{\gamma\in\SL_2(\Z): \gamma\equiv\begin{pmatrix}
    1&0\\
    0&1
\end{pmatrix}\hspace{-2ex}\pmod{p}\Big\}\] acts on the complex upper half plane
\[\mathfrak{H}=\{z\in \C\mid \Im z>0\}\]
via Möbius transformations, and we denote the quotient space by $Y(p)=\varGamma(p)\backslash \mathfrak{H}$. The latter is endowed with the structure of a complex manifold in the obvious way. We should remark, however, that carrying out the construction of this \enquote*{obvious} complex structure on $Y(p)$ can be a little involved; this is due to the phenomenon known as \textit{ellipticity}, which refers to the existence of points on $\mathfrak{H}$ with large stabiliser. Nevertheless, $Y(p)$ is naturally a Riemann surface, whose compactification $X(p)$ is a complex projective curve known as the \textit{principal
modular curve of level $p$.} For an authoritative treatment of this, we refer to \cite[Chapter 2]{diamond2005first}.

To descend into the world of number fields, one proceeds as follows. There is a natural action of $\SL_2(\F_p)\cong \SL_2(\Z)/\varGamma(p)$ on $X(p)$ via Möbius transformations, and the latter gives rise to a Galois cover
\[X(p)\to \mathbb{P}^1(\C)\] with group $\PSL_2(\F_p)$ \cite[§7.5]{diamond2005first}. Remarkably, this map can be realised over $\Q(p)$. Fix a choice of a quadratic non-residue $v\in\F_p^\times$, and put
\[V=\begin{pmatrix}
    0&v\\-1&0
\end{pmatrix}\in \GL_2(\F_p).\]
The subgroup \[ \F_p^\times I\cup\F_p^\times V\subseteq\GL_2(\F_p)\] then defines a model for the modular curve $X(p)$ over $\Q$. The relevant construction is detailed in \cite[§IV.3]{deligne1973schemas}; see in particular the discussion given in paragraph 3.18. The model over $\Q$ will simply be denoted by $X(p).$ The base-change of $X(p)$ to $\Q(p)$ defines a Galois extension of the function field of $\mathbb{P}^1/\Q(p)$ with group $\PSL_2(\F_p)$. 

The natural action of the absolute Galois group $G_\Q$ on $X(p)$ induces an action of $G_\Q$ on $\Aut_{\overline \Q}(X(p))\cong \PSL_2(\F_p)$ via conjugation. (Note that the latter identification is where the assumption that $p\geq 7$ comes in; see the appendix to Part I of \cite{mazur1998open}.) This {inner} action is described by the homomorphism
\[\eta:G_\Q\to \Gal(\Q(p)/\Q)\xrightarrow{\sim}\langle V\rangle\to \PGL_2(\F_p).\] A proof of this can be found in \cite[Proposition 2.1]{10.1112/S0024609304004187}.

Let $F$ be a number field. Fix a projective Galois representation
\[\overline\rho: G_F\to\PGL_2(\F_p)\] with cyclotomic determinant, and put $\overline\rho^\vee(\sigma)=\overline\rho(\sigma^{-1})^t$ for the \textit{contragredient} of $\overline\rho$. Since $\det(\eta)=\overline\varepsilon_p^\pr$, the rule $$\xi:\sigma\mapsto\overline\rho^\vee(\sigma)\eta(\sigma)$$ is a 1-cocycle of $G_F$ with values in $\PSL_2(\F_p)$, and hence descends to a cohomology class
    in $H^1(G_F,\PSL_2(\F_p))$. The latter then defines a twist $X_{\overline\rho}(p)$ of $X(p)$ over $F$; see \cite[§X.2]{silverman2009arithmetic}. Its non-cuspidal, non-CM $F$-rational points classify pairs $(E,s)$ where $E$ is a non-CM elliptic curve over $F$ and $s:E[p]\cong\F_p^{\oplus 2}$ is a symplectic isomorphism (i.e., one that respects the Weil and the determinant parings) such that if $\{T_1,T_2\}$ is the basis for $E[p]$ defined by $s$ then the projective Galois representation
$$\overline\rho^\pr_{E,p}:G_F\to\PGL_2(\F_p)$$ 
corresponding to $\{T_2,T_1\}$ coincides with $\overline\rho$ \cite[Theorem 2.1]{10.1112/S0024609304004187}.

\begin{ex}
   If $E/F$ is an elliptic curve then we may consider the twist $X_E^\pr(p)$ of $X(p)$ by the projective representation $\overline\rho^\pr_{E,p}$ attached to its $p$-torsion subgroup.
\end{ex}

    The principal concern of this paper is to determine when $X_{\overline\rho}(p)$ is in fact defined over $\Q$. 
    Our main result is the following. 
\begin{thm}\label{thm1}
The curve $X_{\overline\rho}(p)$ is defined over $\Q$ if and only if $\overline \rho$ extends to a representation $\overline r:G_\Q\to\PGL_2(\F_p)$ with cyclotomic determinant.
\end{thm}
We would like to draw the reader's attention to an instance of Theorem \ref{thm1} which is particularly relevant to the discussion given in Section \ref{1.1}. 

\begin{thm}\label{thm2} Suppose that $F$ is a cyclic number field of degree $d$, and pick a $\tau\in G_\Q$ which restricts to a generator of $\Gal(F/\Q).$ Let $E/F$ be an elliptic curve. Then, the twist $X_E^\pr(p)$ of $X(p)$ by $\overline\rho^\pr_{E,p}$ is defined over $\Q$ if and only if there exists a continuous character $\overline\chi: G_F\to \{\pm1\}\subseteq\F_p^\times$ and a symplectic isomorphism
\[\varphi:\tau(E)[p]\xrightarrow{\sim} E[p]\otimes\overline{\chi}\] of $G_F$-representations
such that $(\varphi\circ\tau)^d=\pm\tau^d$ as maps on $E[p]$.

If $F$ is an imaginary quadratic field then, choosing $\tau$ to be a complex conjugation, the above condition on $\varphi\circ\tau$ is equivalent to: $\varphi\circ\tau$ is represented by
$$\begin{pmatrix}
        1&0\\0 &-1
    \end{pmatrix}\,\,\,\textit{or}\,\,\,\begin{pmatrix}
        i&0\\0 &i
    \end{pmatrix}$$
with respect to some basis for $E[p]$, where $i\in\F_p^\times$ has $i^2=-1$ (if such an $i$ exists).
\end{thm}
\subsection{Outline of strategy} It is evident that the twisted curve $X_{\overline \rho}(p)$ being defined over $\Q$ is equivalent to the cocycle $\xi$ lying in the image of the restriction map
\[\Res^{G_\Q}_{G_F}:H^1(G_\Q,\PSL_2(\F_p))\to H^1(G_F,\PSL_2(\F_p)),\]
and this is invariably how we will approach the problem.

Before explaining how the proof of Theorem \ref{thm1} proceeds, let us first dabble in some more terminology. Given a subfield $K$ of $F$, we will say that $\overline\rho$ is \textit{invariant} with respect to the extension $F/K$ if the following holds: for every element $\tau\in G_{K}$, there is some $g_\tau\in \PGL_2(\F_p)$ such that $$\overline\rho(\tau\sigma\tau^{-1})=g_\tau \overline\rho(\sigma)g_\tau^{-1}$$ for all $\sigma\in G_F$. If in addition we can ensure that $\det(g_\tau)=\overline\varepsilon_p^\pr(\tau)$ then $\overline\rho$ will be referred to as being \textit{strongly invariant} with respect to $F/K$. Note that in order to check that $\overline\rho$ is (strongly) invariant with respect to $F/K$, it suffices to find suitable $g_{\tau}$ for $\tau$ running through a set of representatives for the generators of $\Gal(F/K).$ Moreover, $\overline\rho$ will be referred to as being \textit{(strongly) compatible} with $F/K$ if it is (strongly) invariant with respect to $F/K$ and if in addition we have
\[\overline\rho(\tau^{d})=g_\tau^{d}\] for all $\tau\in G_K$,
where $d=d_\tau$ is the least positive integer for which $\tau^d\in G_F$. The proof of Theorem \ref{thm1} begins (or rather, ends) with the following auxiliary result.
\begin{lmm}\label{lmm1}
    Let $L/K$ be a cyclic field extension. Then, a Galois representation $\overline r:G_L\to \PGL_2(\F_p)$ with cyclotomic determinant extends to $G_K$ if and only if $\overline r$ is compatible with $L/K$. The extension has cyclotomic determinant precisely when $\overline r$ is strongly compatible with $L/K$.
\end{lmm}

Once Lemma \ref{lmm1} has been proved, it remains to relate the question of definability of $X_{\overline\rho}(p)$ over $\Q$ with some sort of compatibility statement. The main obstruction here is that $F/\Q$ is not necessarily assumed to be cyclic, but one can get away with considering the lattice of field extensions
$$
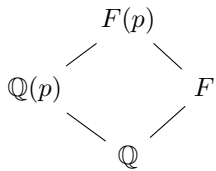
\begin{tikzcd}[row sep = small, column sep = tiny]
    &\arrow[dash]{dr}\arrow[dash]{dl}F(p)&\\\Q(p)\arrow[dash]{dr}&&F\arrow[dash]{dl}\\&\Q&
\end{tikzcd}
$$ where we put $F(p)$ for the compositum of $F$ and $\Q(p)$. The precise results we are alluding to are given below. Before being able to state them, however, we require one last bit of notation. Let $c_p$ be a lift of the generator of $\Gal(\Q(p)/\Q)$ to $G_\Q.$ If $\Q(p)\not\subseteq F$ then $F$ and $\Q(p)$ are linearly disjoint over $\Q$, since $\Q(p)/\Q$ has degree 2. In that case, seeing as restriction gives an isomorphism
\[\Gal(F(p)/F)\xrightarrow{\sim}\Gal(\Q(p)/\Q),\] we may (and do) choose our $c_p$ to lie in $G_F$.
\begin{prop}\label{prop3}
    Suppose $\Q(p)\not\subseteq F$. Then, in order that $X_{\overline\rho}(p)$ be defined over $\Q$, it is necessary and sufficient that there exist an extension $\overline r: G_{\Q(p)}\to\PSL_2(\F_p)$ of $\overline\rho|_{G_{F(p)}}$ such that
 $$\overline r(c_p^{-1}\sigma c_p)=\overline\rho(c_p)^{-1}\overline r(\sigma)\overline\rho(c_p)$$ for all $\sigma\in G_{\Q(p)}$.
\end{prop}
The special case when $\Q(p)\subseteq F$ is admittedly a bit clumsy, due to the fact that it is impossible to choose our $c_p$ to lie in $G_F$.
\begin{prop}\label{prop4}
 Suppose $\Q(p)\subseteq F$. Then, in order that $X_{\overline\rho}(p)$ be defined over $\Q$, it is necessary and sufficient that there exist an element $g\in\PGL_2(\F_p)\setminus\PSL_2(\F_p)$ and an extension
 $\overline r: G_{\Q(p)}\to\PSL_2(\F_p)$ of $\overline\rho$ such that $\overline r(c_p^2)=g^2$ and $$\overline r(c_p^{-1}\sigma c_p)=g^{-1}\overline r(\sigma)g$$ for all $\sigma\in G_{\Q(p)}.$
\end{prop}
Theorem \ref{thm1} then follows from Propositions \ref{prop3} and \ref{prop4}, by applying Lemma \ref{lmm1} to the cyclic extension $\Q(p)/\Q$.
\subsection{Applications} We conclude this introductory section with some examples illustrating the possible applications of Theorem \ref{thm1}. 

\begin{ex} Let $F$ be a number field with normal closure $K$. Suppose we are given a $\Q$-curve $E/F$, i.e., an elliptic curve whose $G_\Q$-conjugates are all $\overline \Q$-isogenous. We contend that, as long as $E$ does not admit complex multiplication, there is an infinite number of rational primes $p\geq 7$ such that $X_E^\pr(p)$ is defined over $\Q$.

Choose an isogeny 
        $$f_\tau:\tau(E)\to E$$ for each $\tau\in \Gal(K/\Q)$. By the Chebotarev density theorem, a positive proportion of all primes $p\geq 7$ have the property that each $\deg(f_\tau)$ is a nonzero square in $\F_p$. Fix such a $p$. Then, we can consider the assignment
    \[\overline r:G_\Q\to \End_{\F_p}(E[p]),\quad \sigma\mapsto f_\sigma\circ\sigma,\] where we write $f_\sigma=f_{\sigma|_K}$ to ease notation. First, note that every map of type $\overline r(\sigma)$ is a linear automorphism; for if we let $$\widehat f_\sigma:E\to \sigma(E)$$ be the isogeny dual to $f_\sigma$ then the fact that $\deg(f_\sigma)\not\equiv0\pmod{p}$ implies that
    \[\overline r(\sigma)^{-1}={\deg(f_\sigma)}^{-1}\cdot(\sigma^{-1}\circ \widehat f_\sigma)\] is a well-defined element of $\End_{\F_p}(E[p])$, inverse to $\overline r(\sigma)$ by construction. Now, given $\sigma,\sigma\in G_\Q$, the self-isogeny $f_\sigma\circ\sigma\circ f_{\sigma'}\circ \sigma^{-1}\circ \widehat f_{\sigma\sigma'}\in \End_{\overline \Q}(E)$ must be given by multiplication by an element of $\Z$ (since $E$ is not a CM curve), so 
\[\overline r(\sigma)\overline r(\sigma')\overline r(\sigma\sigma')^{-1}=\deg(f_{\sigma\sigma'})^{-1}\cdot (f_\sigma\circ\sigma\circ f_{\sigma'}\circ \sigma^{-1}\circ \widehat f_{\sigma\sigma'})\in \Aut_{\F_p}(E[p])\] becomes trivial upon projectivisation. We thus obtain a projective representation $\overline r^\pr$ of $G_\Q$ from $\overline r$ which restricts to $\bar\rho^\pr_{E,p}$ on $G_F$, again in view of the absence of complex multiplication (or, we could have taken $f_{\tau}=\id_E$ whenever $\tau\in \Gal(K/F)$). In order that we may safely apply Theorem \ref{thm1}, it remans to check that $\overline r^\pr$ has cyclotomic determinant. But since $f_\sigma$ and $\widehat f_\sigma$ are adjoint for to the Weil pairing \cite[Proposition III.8.2]{silverman2009arithmetic}, we have
\[\det(\overline r(\sigma))=\deg(f_\sigma)\overline\varepsilon_p(\sigma),\] and the latter coincides with $\overline \varepsilon_p(\sigma)$ up to an element of $\F_p^{\times2}$ by our choice of $p$.
\end{ex}
\begin{ex} We presently consider Theorem \ref{thm2} in the context of quadratic twists. Fix a cyclic extension $F/\Q$ of degree $d$ and let $\tau\in G_\Q$ restrict to a generator of $\Gal(F/\Q).$ Consider the elliptic curve 
\[E:\lambda y^2=x^3+ax+b\] where $\lambda\in F^\times$ and $a,b\in \Q$. If $\delta\in\overline\Q$ is a square-root of $\tau(\lambda)/\lambda$ then $(x,y)\mapsto (x,\delta y)$ gives an isomorphism $f:\tau(E)\to E$ identifying the action of $G_F$ on $\tau(E)$ with the action of $G_F$ on $E$ twisted by the character
 $\overline\chi:G_F\to \{\pm1\}$ associated to $F(\delta)/F$. A standard argument shows that
     \[(\delta\,\tau(\delta)\,\cdots\,\tau^{d-1}(\delta))^2=1.\] This readily implies that $(f\circ\tau)^d=\pm\tau^d$ holds on $E$. Putting everything together, we find that $X_E^\pr(p)$ is defined over $\Q$ for any rational prime $p\geq 7$. 

\end{ex}
\begin{ex}
Suppose once more that $F$ is a cyclic extension of $\Q$. Let $$\overline\rho:G_F\to\PGL_2(\F_p)$$ be an absolutely irreducible representation, strongly invariant with respect to $F/\Q$. The twist $X_{\overline\rho}(p)$ is then defined over $\Q$.

    Indeed, let $\tau\in G_\Q$ be so that $\Gal(F/\Q)=\langle\tau|_{F}\rangle$. Choose a corresponding $g_{\tau}$ as in the definition of strong invariance above. If $F/\Q$ has degree $d\geq1$ then $\tau^d\in G_F$, whence $\overline\rho(\tau^d)\overline\rho(\sigma)\overline\rho(\tau^d)^{-1}=g_\tau^d\overline\rho(\sigma)g_{\tau}^{-d}$ for $\sigma\in G_F$. So, $g_\tau^{-d}\overline\rho(\tau^d)$ commutes with the image of $\overline\rho$ pointwise, and hence must be trivial by absolute irreducibility and Schur's lemma for projective representations  \cite[Lemma 2.1]{CHENG2015230}. In other words, $\overline\rho$ is strongly compatible with $F/\Q$, and the results above apply.
\end{ex}

\section{Proofs}\subsection{Proof of Lemma \ref{lmm1}} A result to the effect of Lemma \ref{lmm1} seems to be well-known in the literature, but concrete references are rare. We therefore give a quick proof.

It is clear that a projective representation of $G_L$ which admits an extension along the inclusion $G_L\hookrightarrow G_K$ is compatible with $L/K$. Assume conversely that $L/K$ is cyclic of degree $d$ and that $\overline r$ is strongly compatible with $L/K$. Let $\tau\in G_K$ restrict to a generator of $\Gal(L/K)$. We have the coset decomposition
\[G_K=G_L\sqcup\tau G_L\sqcup\cdots\sqcup\tau^{d-1}G_L,\] and we use it to define an extension $\overline R:G_K\to \PGL_2(\F_p)$ in the obvious way; namely, we let $\overline R(\tau^n\sigma)=g_\tau^n\overline r(\sigma)$ for any integer $n$ and any $\sigma\in G_L$. To see that this is well-defined, suppose that $\tau^n\sigma=\tau^{m}\sigma'$ for integers $m,n$ and  $\sigma,\sigma'\in G_L$. Then, $n-m$ must be a multiple of $d$, say it is equal to $dk$ for some integer $k$, and as such we must have $\sigma'=\tau^{dk}\sigma$. Therefore,
\[g_\tau^m\overline r(\sigma')=g_\tau^m\overline r(\tau^d)^k\overline r(\sigma)=g_\tau^{m+dk}\overline r(\sigma)=g_\tau^{n}\overline r
(\sigma)\] using compatibility. It thus remains to check that $\overline R$ is a homomorphism. Let $m,n$ be integers and let $\sigma,\sigma'\in G_L$. Since $G_L$ is normal in $G_K$, we may write $$\tau^n\sigma\tau^m\sigma'=\tau^{m+n}(\tau^{-m}\sigma\tau^m)\sigma'$$ with the two rightmost factors belonging in $G_L$. As such,
\[\overline R(\tau^n\sigma\tau^m\sigma')=g_\tau^{m+n}\overline r(\tau^{-m}\sigma\tau^m)\overline r(\sigma')=g_\tau^{m+n}g_\tau^{-m}\overline r(\sigma)g_\tau^m\overline r(\sigma')=\overline R(\tau^n\sigma)\overline R(\tau^m\sigma')\] by the invariance of $\overline r$ with respect to $L/K$, as required. The statement about the determinant of $\overline R$ is clear.
\subsection{Proof of Propositions \ref{prop3} and \ref{prop4}} Let us first relate the restriction map on cohomology with extensions of cocycles.

\begin{lmm}
    Let $G$ be a group, $\varGamma\subseteq G$ a subgroup and $X$ a (possibly nonabelian) group equipped with an action of $G.$ Consider a cohomology class $[\pi]\in H^1(\varGamma,X)$ represented by a 1-cocycle $\pi:\varGamma\to X$. Then, $[\pi]$ lies in the image of
    \[\Res^G_{\varGamma}:H^1(G,X)\to H^1(\varGamma,X)\]
    if and only if $\pi$ extends to a 1-cocycle $\widetilde\pi:G\to X$.
\end{lmm}
\begin{proof}
    Suppose $[\pi]=[\pi']$ for a 1-cocycle $\pi':G\to X$. This means that there is some $x\in X$ such that $\pi(\gamma)=x^{-1}\pi'(\gamma){\,}^\gamma x$ for all $\gamma\in \varGamma$. We then obtain an extension of $\pi$ to $G$ on setting $\widetilde\pi(g)=x^{-1}\pi'(g){\,}^gx$ for $g\in G$. Converse is clear.
\end{proof}

In light of the above lemma, we need only consider the question of whether $\xi$ extends to a cocycle on $G_\Q$. Let us treat the case $\Q(p)\not\subseteq F$ first. By construction, we have a coset decomposition
\[G_F=G_{F(p)}\sqcup c_p G_{F(p)}\] and a similar one with \enquote*{$F$ replaced by $\Q$}. This then gives the more explicit equations
\begin{equation}\label{1}\xi(\sigma)=
    \begin{cases}
        \overline\rho^\vee(\sigma)&\text{if }\sigma\in G_{F(p)};\\\overline\rho^\vee(c_p)\overline\rho^\vee(c_p^{-1}\sigma)V&\text{if }\sigma\notin G_{F(p)};
    \end{cases}
\end{equation} for $\sigma\in G_{F}$; recall that, in Section \ref{1.2}, we defined
\[V=\begin{pmatrix}
    0&v\\-1&0
\end{pmatrix}\]
with $v\in \F_p^\times$ a fixed quadratic non-residue.

The task at hand is to determine the precise conditions under which there is a cocycle $x:G_\Q\to\PSL_2(\F_p)$ extending $\xi$. Of course, such an $x$ will verify the conditions in \eqref{1} for $\sigma\in G_{F(p)}$. Noting that the action of $G_{\Q(p)}$ on $\PSL_2(\F_p)$ is trivial, we thus see that $\overline r=x^\vee|_{G_{\Q(p)}}$ must necessarily be a representation with cyclotomic determinant extending $\overline\rho|_{G_{F(p)}}$. So, we have to determine when such an $\overline r$ extends to a suitable cocycle $x:G_\Q\to \PSL_2(\F_p).$ Using a coset decomposition similar to the one given for $G_F$ above, it becomes clear that such an $x$ is uniquely determined by the value it takes on $c_p$. More concretely, if $x:G_\Q\to \PSL_2(\F_p)$ is a cocycle agreeing with $\overline r$ on $G_{\Q(p)}$ and taking the value $x_p$ at $c_p$ then the cocycle condition shows that $x(c_p\sigma)=x_pV\overline r^\vee(\sigma)V$ for $\sigma\in G_{\Q(p)}$.

However, \eqref{1} supplies us with a value for $x_p$ since $c_p\in G_F$; namely, we have no other choice but to take $x_p=\overline\rho^\vee(c_p)V$. Defining a function $x:G_\Q\to \PSL_2(\F_p)$ by \begin{equation}
    x(\sigma)=\label{2}\begin{cases}
    \overline r^\vee(\sigma)&\text{if }\sigma\in G_{\Q(p)};\\\overline\rho^\vee(c_p)\overline r^\vee(c_p^{-1}\sigma)V&\text{if }\sigma\notin G_{\Q(p)},
\end{cases}
    \end{equation} it remains to see if $x$ is a cocycle agreeing with $\xi$ on $G_F$. Comparing \eqref{1} and \eqref{2}, that $x$ agrees with $\xi$ on $G_{F}$ follows under no further assumptions by using the fact that $\overline r$ restricts to $\overline \rho$ on $G_{F(p)}$. Let us now consider the cocycle condition for our $x$. Pick $\sigma,\tau\in G_\Q$. Assume first that $\sigma\in G_{\Q(p)}$. Then we need $x(\sigma\tau)=\overline r^\vee(\sigma)x(\tau)$. This is clear if $\tau\in G_{\Q(p)}$ since $\overline r$ is a homomorphism, and if $\tau\notin G_{\Q(p)}$ then this amounts to saying that
\[\overline\rho^\vee(c_p)\overline r^\vee(c_p^{-1}\sigma\tau)V=\overline r^\vee(\sigma)\overline \rho^\vee(c_p)\overline r^\vee(c_p^{-1}\tau)V.\] But then, $c_p^{-1}\sigma\tau=(c_p^{-1}\sigma c_p)(c_p^{-1}\tau)$ with both factors belonging to $G_{\Q(p)}$, so the above condition yields\begin{equation}
    \label{3}\overline r(c_p^{-1}\sigma c_p)=\overline\rho(c_p)^{-1}\overline r(\sigma)\overline\rho(c_p)
\end{equation} for $\sigma\in G_{\Q(p)}$ upon rewriting. The remaining case of $\sigma\notin G_{\Q(p)}$ is similar, and in the end also gives condition \eqref{3}. This establishes Proposition \ref{prop3}.

We give a few remarks to indicate how the argument modifies when $\Q(p)\subseteq F$. In that case, condition \eqref{1} says only that $\xi$ is a homomorphism $G_{F}\to \PSL_2(\F_p)$. If $x$ is a cocycle extending $\xi$ to $G_\Q$ then $x|_{G_{\Q(p)}}=\overline r^\vee$ for $\overline r$ a representation of $G_{\Q(p)}$ extending $\overline\rho$, and as before $\overline r$ determines $x$ once a value $x_p$ of $x$ at $c_p$ has been fixed. However, since $c_p$ can no longer be assumed to lie in $G_F$, an \enquote*{arbitrary} choice of $x_p\in \PSL_2(\F_p)$ has to be made. Then, using the coset decomposition
\[\PGL_2(\F_p)=\PSL_2(\F_p)\sqcup \PSL_2(\F_p)V\] we may write $x_p=(g^{-1})^tV$ for a unique $g\in\PGL_2(\F_p)$ not lying in $\PSL_2(\F_p)$. Defining a function $x:G_\Q\to \PGL_2(\F_p)$ as in \eqref{2} with $(g^{-1})^t$ in place of $\overline\rho^\vee(c_p)$, the subsequent argument goes through verbatim and we see that an $x$ thus defined satisfies the required properties if and only if our choice of $g$ verifies:
\begin{itemize}
\item $\overline r(c_p^2)=g^2$;
    \item $\overline r(c_p^{-1}\sigma c_p)=g^{-1}\overline r(\sigma)g$ for all $\sigma\in G_{\Q(p)}$.
\end{itemize}
We note that the former property comes from considering the cocycle condition; indeed, using the fact that $c_p^2\in G_{\Q(p)}$, we obtain
\[(g^{-2})^t=x(c_p)Vx(c_p)V=x(c_p){\,}^{c_p}x(c_p)=x(c_p^2)=\overline r^\vee(c_p^2).\] Proposition \ref{prop4} follows.
\subsection{Proof of Theorem \ref{thm2}} We assume for the remainder of this article that $F$ is a cyclic extension of $\Q$, which will later be specified to an imaginary quadratic field. Fix an automorphism $\tau\in G_\Q$ restricting to a generator of $\Gal(F/\Q)$. Let $E/F$ be an elliptic curve equipped with a symplectic identification $E[p]\cong \F_p^{\oplus 2}.$ Put $X_E^\pr(p)$ for the twist of $X(p)$ by the projective representation $\overline\rho^\pr_{E,p}$.

In view of Theorem \ref{thm1} and Lemma \ref{lmm1}, the question of whether $X_E^\pr(p)$ is defined over $\Q$ reduces to the question of whether we can find an element $g\in\PGL_2(\F_p)$ satisfying the following conditions:
\begin{itemize}
    \item $\overline\rho_{E,p}^\pr(\tau\sigma\tau^{-1})=g\overline\rho_{E,p}^\pr(\sigma)g^{-1}$ for all $\sigma\in G_F$;
    \item $g^d=\overline\rho_{E,p}^\pr(\tau^d)$ where $d=(F:\Q);$
    \item $\det(g)=\overline\varepsilon_p^\pr(\tau)$.
\end{itemize}
Consider a lift $h\in\GL_2(\F_p)$ of such a $g$ such that $\det(h)=\overline\varepsilon_p(\tau)$. Then, there is a scalar $\lambda\in\F_p^\times$ for which $h^d=\lambda\overline\rho_{E,p}(\tau^d)$. Taking determinants here yields $\lambda=\pm1.$ Since $\tau$ identifies the action of $\sigma\in G_F$ on $\tau(E)[p]$ with the action of $\tau\sigma\tau^{-1}$ on $E[p]$, the first condition above, when lifted to $\GL_2(\F_p)$, reads
\[\overline\rho_{\tau(E),p}(\sigma)=h\overline\chi(\sigma)\overline\rho_{E,p}(\sigma)h^{-1}\] for a function $\overline\chi:G_F\to \F_p^\times,$ which is immediately seen to be a continuous character. Comparing determinants in the above equation gives
\[\overline\varepsilon_p(\sigma)=\overline\chi(\sigma)^2\overline\varepsilon_p(\sigma),\]
whence it follows that $\overline \chi(G_F)\subseteq \{\pm1\}$. This is the first part of Theorem \ref{thm2}.

Explicitly, if $h$ represents a map $\phi:E[p]\to E[p]$ then $\varphi=\phi\circ\tau^{-1}$ identifies the action of $G_F$ on $\tau(E)[p]$ with that on $E[p]\otimes\overline\chi$, and comparing the Weil pairings shows that the determinant condition on $h$ is equivalent to $\varphi$ being symplectic. Finally, $h^d=\pm\overline\rho_{E,p}(\tau^d)$ in $\GL_2(\F_p)$ translates to $(\varphi\circ\tau)^d=\pm\tau^d$ on $E[p]$.

Now, suppose $F$ is an imaginary quadratic field. We may then arrange for $\tau$ to be a complex conjugation, in which case $\tau^2=1$ and $\overline\varepsilon_p(\tau)=-1$. In particular, there are two possibilities: $h^2=1$ or $h^2=-1$. If the former holds, then $h$ is diagonalisable and has eigenvalues $\pm1$, so it is similar to \[\begin{pmatrix}
    1&0\\0&-1
\end{pmatrix}\] in view of $\det(h)=-1$. Suppose the latter holds. Were $x^2+1$ irreducible over $\F_p$ (which happens when $p\equiv-1\pmod{4}$) then putting $h$ into Frobenius normal form would yield $\det(h)=1$, contradiction. If otherwise $\F_p$ admits a square-root of $-1$, call it $i$, then as above we deduce that $h$ is similar to 
\[\begin{pmatrix}
    i&0\\0&i
\end{pmatrix}.\]
The proof of Theorem \ref{thm2} is complete.
\section*{Acknowledgements} The author would like to thank James Newton for suggesting the problem of descent for $X_{\overline\rho}(p)$, for many valuable discussions around the subject, and indeed for several helpful comments on an earlier draft. The presentation of this article was improved by the anonymous referee's kind suggestions, for which
the author is most grateful. The research presented in this paper was partially supported by a grant from St Catherine's College, Oxford.
\bibliography{bib.bib}
\bibliographystyle{amsalpha}
\,
\end{document}